\newif\ifmyfig \myfigtrue
\newif\ifmytab \mytabtrue

\iffalse%\iftrue%

\documentclass{article}

\else 

\documentclass[12pt]{article}

\fi

\usepackage{url}
\usepackage{amsmath}
\usepackage{amsthm}
\usepackage{graphicx}

\newtheorem{theorem}{Theorem}
\newtheorem{proposition}[theorem]{Proposition}
\newtheorem{lemma}[theorem]{Lemma}
\theoremstyle{definition}
\newtheorem{example}{Example}

\def\pd#1{\partial_{#1}}

\title{
Holonomic gradient method for the probability content of a simplex region
under a multivariate normal distribution
}
\author{Tamio Koyama}
\date{}

\begin{document}
\maketitle

\begin{abstract}
We use the holonomic gradient method
to evaluate the probability content of a simplex region
under a multivariate normal distribution.
This probability equals to the integral of the probability density function 
of the multivariate Gaussian distribution on the simplex region.
For this purpose, we generalize the inclusion--exclusion 
identity which was given for polyhedra, to the faces of a polyhedron.
This extended inclusion--exclusion identity enables us to
calculate the derivatives of the function associated with
the probability content of a polyhedron in general position.
We show that these derivatives can be written as integrals of
the faces of the polyhedron.

\noindent{\bf keyword\/:}
Holonomic Gradient Method,
Inclusion--Exclusion Identity, %Simplex Region,
Tukey--Kramer Studentized Range
\end{abstract}

%!!! BEGIN !!!
\section{Introduction}
The holonomic gradient method (HGM) is an algorithm for the numerical calculation 
of holonomic functions. 
It is a variation on the holonomic gradient descent (HGD) proposed in 
\cite{NNNOSTT2011}.
A holonomic function is an analytic function of several variables
which satisfies a holonomic system.
Here, a holonomic system refers to a system of linear differential equations 
with polynomial coefficients which induces a holonomic module in terms of 
$D$-module theory \cite{oaku1994}. 
The HGM evaluates a holonomic function by numerically solving 
an initial value problem for an ordinary differential equation.
This ordinary differential equation is derived from the Pfaffian equation
(an integrable connection) associated with the function. 
For details, see \cite{Koyama2015} and its references.
Several  normalizing constants and the probability content of a region
can be regarded as a holonomic function with respect to their parameters,
and we can use the HGM to evaluate the solution to this function.
For example, the HGM was used to evaluate the cumulative distribution function 
for the largest root of the Wishart matrix in \cite{HNTT2013},
and we utilized the HGM 
for the orthant probability and the Fisher--Bingham integral
in \cite{2015koyama-takemura1} and \cite{2015koyama-takemura2}, respectively.

In this paper, we apply the HGM to the numerical calculation of
the probability content of a polyhedron with a multivariate normal distribution.
A polyhedron is a subset of a $d$-dimensional Euclidean space $\mathbf R^d$
which is defined by a finite number of linear inequalities.
%Intervals of real numbers, orthants, cubes, and simplices are 
%examples of polyhedra.
In \cite{naiman-wynn1992}, Naiman and Wynn described examples of
how the evaluation of the probability content of a polyhedron can be used 
to find critical probabilities for multiple comparisons.
The numerical calculation of the probability content is discussed 
in \cite{mkh2012}.

The probability content of a polyhedron is a generalization of 
orthant probabilities discussed in 
\cite{plackett1954}, \cite{mkh2003},
and \cite{2015koyama-takemura1},
since the orthant probabilities can be expressed as the probability content
of a simplicial cone.
We derived the HGM for orthant probabilities in \cite{2015koyama-takemura1},
and its implementation in \cite{knost2014}.
A study of phylogenomics utilized our implementation 
in \cite{weyenberg-howe-yoshida2015}.

In order to utilize the HGM for the numerical calculation of
the probability content of a polyhedron,
we need to provide the Pfaffian equation explicitly and
to evaluate the initial value for the ordinary differential equation.
Our previous paper \cite{Koyama2015} showed that
the probability content of a polyhedron in general position can be expressed as
an analytic function and we explicitly provided a holonomic system 
and Pfaffian equations for this function.
This analytic function is an interesting special case
of the integral considered in \cite{aomoto-kita-orlik-terao}, 
and the Pfaffian equations is a generalization of
a differential-recurrence formula for the orthant probability
in \cite{plackett1954}.

In this paper, we show how to calculate the derivatives of the function,
and prove that these derivatives can be written as integrals of
the faces of the polyhedron.
This result provides formulae to compute the initial value exactly
for the cases where the polyhedron is in general position and bounded,
 or the polyhedron is a simplicial cone.
In order to calculate the derivatives, 
we generalize the inclusion--exclusion identity that was given for polyhedra
in \cite{edelsbrunner}, to the faces of the polyhedron.
Since a face of a polyhedron is also a polyhedron, 
the inclusion--exclusion identity for the face holds, i.e., 
the indicator function of the face can be written as 
a linear combination of indicator functions of simplicial cones.
Our generalized inclusion--exclusion identity gives 
an explicit expression for this linear combination.

In the numerical experiments, 
%we consider simplices and simplicial cones, 
%as these are fundamental examples of polyhedra.
utilizing the theoretical results concerned with 
the Pffafian equation and the initial value, 
we implement the HGM to evaluate the probability contents
of a simplex and a simplicial cone.
We show that our implementation works well 
for a 10-dimensional simplex.

This paper is organized as follows. 
In section \ref{sec:prev} 
we review results from our previous paper \cite{Koyama2015}.
In section \ref{sec:I-E}
we extend the inclusion--exclusion identity
which was given for polyhedra in \cite{edelsbrunner}, 
and provide an analogous formula for the indicator function of a face of 
a polyhedron.
In section \ref{sec:diff}
we calculate the derivatives of 
the function defined by the probability content of a polyhedron 
for the multivariate normal distribution, 
and show that these derivatives can be written by integrals on corresponding faces.
In section \ref{sec:HGM}
attention is directed towards
the case where the polyhedron in general position is bounded,
and the case where the polyhedron is a simplicial cone.
We discuss the evaluation of the multivariate normal probabilities 
of polyhedra by the HGM in these two cases.
In section \ref{sec:NE} 
we present numerical examples.

{\bf Acknowledgements.}\/
This work was supported by JSPS KAKENHI Grant Number 263125.

\section{Summary of Previous Work}\label{sec:prev}
In this section we review the results of our previous paper \cite{Koyama2015}.
Let us consider a polyhedron
\begin{equation}\label{defP}
P := 
\left\{
x\in \mathbf R^d : 
\sum_{i=1}^d \tilde a_{ij}x_i+\tilde b_j \geq 0,\, 1\leq j \leq n 
\right\}
\end{equation}
where $\tilde a_{ij}, \tilde b_j\,(1\leq i\leq d,\,1\leq j\leq n)$ are real numbers.
We denote by $\tilde a$ and $\tilde b$ the $d\times n$ matrix $(\tilde a_{ij})$ 
and the vector $(\tilde b_1,\dots, \tilde b_n)^\top$ respectively.
We suppose that 
the polyhedron $P$ is in general position and 
 its bounding half-spaces are 
\begin{equation}\label{bd-spaces}
H_j := 
\left\{x\in\mathbf R^d: \sum_{i=1}^d \tilde a_{ij}x_i+\tilde b_j \geq 0 \right\}
\quad \left( 1 \leq j \leq n \right).
\end{equation}
For the definitions of general position and
the set of the bounding half-spaces for a polyhedron, 
see \cite{Koyama2015}.

We denote by $\mathcal F$ the abstract simplicial complex associated with 
the polyhedron $P$, i.e., 
$$
\mathcal F :=
\left\{
J\subset\{1,2,\dots, n\} \mid \bigcap_{j\in J} H_j \neq \emptyset
\right\}.
$$
Let 
\begin{equation}\label{eq:varphi}
\varphi(a,b) = 
\int_{\mathbf R^d}
\frac{1}{(2\pi)^{d/2}}
\exp\left(-\frac{1}{2}\sum_{i=1}^dx_i^2\right)
\sum_{F\in \mathcal F}\prod_{j\in F} \left(H(f_j(a,b,x))-1\right)
dx
\end{equation}
be a function with variables $a_{ij}, b_j\,(1\leq i\leq d,\,1\leq j \leq n)$.
Here, we set $f_j(a,b,x)=\sum_{i=1}^da_{ij}x_i+b_j$
and $H(x)$ is the Heaviside function, i.e., 
$$
H(x):=
\begin{cases}
1 & (x\geq 0)\\
0 & (x  <  0).
\end{cases}
$$
We denote by $a$ and $a_j\,( j=1,\dots n)$ 
the $d\times n$ matrix with elements $a_{ij},(i=1,\dots d)$ and the column vector $(a_{1j},\dots,a_{dj})^\top$ 
respectively.
We denote by $b$ a column vector $(b_1,\dots,b_n)^\top$ with length $n$.
For $J\in\mathcal F$, we put 
\begin{equation}\label{22}
g^J(a,b)= \left(\prod_{j\in J}\partial_{b_j}\right)\bullet\varphi(a,b),
\end{equation}
and let $g(a,b)=(g^J(a,b))_{J\in \mathcal F}$ be a vector-valued function.
Then $g(a,b)$ satisfies the following Pfaffian equations 
\cite[Theorem 22]{Koyama2015} :
\begin{align}
\label{pfaff1} \pd{a_{ij}}g^J &= \sum_{k=1}^na_{ik}\pd{b_k}\pd{b_j}g^J
\quad (1\leq i\leq d, \, 1\leq j\leq n, \, J\in \mathcal F), \\
\label{pfaff2} \pd{b_j} g^J &= g^{J\cup \{j\}} \quad (j\in J^c, \, J\in \mathcal F), \\
\label{pfaff3} \pd{b_j}g^J 
&= -\sum_{k\in J}\alpha^{jk}_J(a)
\left(b_kg^J+\sum_{\ell\in J^c}\alpha_{k\ell}(a)g^{J\cup \ell}\right)
\quad (j\in J, \, J\in\mathcal F).
\end{align}
Here, $(\alpha^{ij}_F(a))_{i, j\in F}$ is the inverse matrix of 
$\alpha_F(a)=\left(\sum_{k=1}^d a_{ki}a_{kj}\right)_{i, j\in F}$, 
which is a sub matrix of the Gram matrix of $a$.
Note that the right hand side of \eqref{pfaff1} can be rewritten, with recourse to
\eqref{pfaff2} and \eqref{pfaff3}, as a linear 
combination of $g^J$ with rational functions as coefficients.

\section{Inclusion-Exclusion Identity for Faces}\label{sec:I-E}
Let $P$ be the polyhedron defined by \eqref{defP}, 
and suppose the family of the bounding half-spaces for $P$ is given by 
\eqref{bd-spaces}.
Then, we have the following inclusion--exclusion identity \cite{edelsbrunner}: 
\begin{equation}\label{19}
\prod_{j=1}^nH\left(\sum_{i=1}^d\tilde a_{ij}x_i+\tilde b_j\right)
=\sum_{J\in\mathcal F}\prod_{j\in J}
\left(H\left(\sum_{i=1}^d\tilde a_{ij}x_i+\tilde b_j\right)-1\right)
\quad (x\in\mathbf R^d).
\end{equation}
In \cite{Koyama2015}, we showed that 
if the polyhedron $P$ is in general position,
there exists a neighborhood $U$ of the parameter 
$(\tilde a,\tilde b)\in\mathbf R^{d\times n}\times\mathbf R^n$
such that 
\begin{equation}\label{20}
\prod_{j=1}^nH\left(\sum_{i=1}^d a_{ij}x_i+ b_j\right)
=\sum_{J\in\mathcal F}\prod_{j\in J}
\left(H\left(\sum_{i=1}^d a_{ij}x_i+ b_j\right)-1\right)
\end{equation}
holds for any $(a,b)\in U$ and $x\in\mathbf R^d$.
The left hand sides of \eqref{19} and \eqref{20} are 
the indicator functions for the corresponding polyhedra.
In this section we give analogous identities for the indicator functions of 
a face of the polyhedra.

Let $\mathcal F$ be the abstract simplicial complex for the polyhedron $P$.
For $J\in\mathcal F$, 
$$
\mathcal F_J:=\{F\in\mathcal F\mid J\subset F\}.
$$
For parameter $(a,b)\in\mathbf R^{d\times n}\times\mathbf R^n$ 
and $J\in\mathcal F$, we define an affine subspace $V(J,a,b)$ by
\begin{equation}\label{h-plane}
V(J,a,b) = 
\left\{
x\in\mathbf R^d \mid \sum_{i=1}^d a_{ij}x_i+ b_j = 0\, (j\in J)
\right\}.
\end{equation}

\begin{proposition}\label{11}
Suppose the polyhedron $P$ is in general position.
For each $J\in\mathcal F$, we have the equation 
\begin{equation}\label{24}
\prod_{j\in [n]\backslash J} H\left(\sum_{i=1}^d \tilde a_{ij}x_i+ \tilde b_j\right)
=\sum_{F\in\mathcal F_J}\prod_{j\in F\backslash J}
\left(H\left(\sum_{i=1}^d \tilde a_{ij}x_i+ \tilde b_j\right)-1\right).
\end{equation}
for any $x\in V(J,\tilde a,\tilde b)$.
\end{proposition}
\begin{proof}
Let $s$ be the number of elements in the set $J$.
Since the polyhedron $P$ is in general position, we have $s\leq d$.
By replacing the indices, we can assume $J=\{n-s+1,\dots,n\}$ without loss of
generality.
Applying the Euclidean transformation for $P$,
we can assume 
$\tilde a_{ij}=0,\,\tilde b_j=0\,(1\leq i\leq d-s,\,n-s+1\leq j\leq n)$.
Then, by the assumption of general position, 
the vectors $\tilde a_{n-s+1},\dots,\tilde a_n$ are linearly independent
(see, \cite[Corollary 20]{Koyama2015}).
Hence we have 
$$
V(J,\tilde a,\tilde b) = \{x\in\mathbf R^d\mid x_{d-s+1}=\cdots=x_d=0\},
$$
and the problem is reduced to the proof of
\begin{equation}\label{9}
\prod_{j=1}^{n-s} H\left(\sum_{i=1}^{d-s} \tilde a_{ij}y_i+ \tilde b_j\right)
=\sum_{F\in\mathcal F_J}\prod_{j\in F\backslash J}
\left(H\left(\sum_{i=1}^{d-s} \tilde a_{ij}y_i+ \tilde b_j\right)-1\right)
\end{equation}
for arbitrary $y=(y_1,\dots,y_{d-s})^\top\in\mathbf R^{d-s}$.
Let us consider a polyhedron 
$$
P' := 
\left\{
y\in \mathbf R^{d-s} \mid
\sum_{i=1}^{d-s} \tilde a_{ij}y_i+\tilde b_j \geq 0,\, 1\leq j \leq n-s 
\right\}.
$$
Suppose that there are $t$ redundant inequalities in the definition of $P'$.
By replacing indices, we can assume that the redundant inequalities are
$
\sum_{i=1}^{d-s} \tilde a_{ij}y_i+\tilde b_j \geq 0,\, (n-s-t+1\leq j \leq n-s).
$
Then, all facets of $P'$ are given by 
$$
F'_j:=P'\cap
\left\{y\in\mathbf R^{d-s}\mid \sum_{i=1}^{d-s} \tilde a_{ij}y_i+\tilde b_j=0\right\}
\quad (1\leq j\leq n-s-t),
$$
and the abstract simplicial complex for $P'$ is 
$$
\mathcal F'=
\left\{
J'\subset\{1,2,\dots, n-s-t\} \mid \bigcap_{j\in J'}F'_j\neq\emptyset
\right\}.
$$
Applying the inclusion-exclusion identity for $P'$, we have 
\begin{equation}\label{10}
\prod_{j=1}^{n-s-t} H\left(\sum_{i=1}^{d-s} \tilde a_{ij}y_i+ \tilde b_j\right)
=
\sum_{F\in\mathcal F'}\prod_{j\in F}
\left(H\left(\sum_{i=1}^{d-s} \tilde a_{ij}y_i+ \tilde b_j\right)-1\right).
\end{equation}

Since the left hand sides of \eqref{9} and \eqref{10} are both equal to
the indicator function of $P'$, they are equal to each other.
Consequently, we need to show that
the right hand sides of \eqref{9} and \eqref{10} are equal.
It is easy to show that the mapping
$$
\psi:\mathcal F'\rightarrow \mathcal F_J
\quad (J'\mapsto J\cup J')
$$
is a bijection. 
Rewriting the right hand side of \eqref{10} in terms of $\mathcal F_J$,
we have the same expression for the right hand side of \eqref{9}.
\end{proof}

\begin{example}
Suppose $d=2$ and $P=H_1\cap H_2\cap H_3$ where
\begin{align*}
H_1 &= \left\{x\in\mathbf R^2\mid -x_1-x_2+1\geq 0\right\}, \\
H_1 &= \left\{x\in\mathbf R^2\mid x_1\geq 0\right\}, \\
H_1 &= \left\{x\in\mathbf R^2\mid x_2\geq 0\right\},
\end{align*}
then the inclusion--exclusion identity \eqref{19} is 
\begin{align}
H\left(-x_1-x_2+1\right)
H\left(x_1\right)
H\left(x_2\right)
&=
1
+\left(H\left(-x_1-x_2+1\right)-1\right) \label{2Mar2020a}\\&
+\left(H\left(x_1\right)-1\right)
+\left(H\left(x_2\right)-1\right) \nonumber\\&
+\left(H\left(x_1\right)-1\right)\left(H\left(x_2\right)-1\right) \nonumber\\&
+\left(H\left(-x_1-x_2+1\right)-1\right)\left(H\left(x_2\right)-1\right) \nonumber\\&
+\left(H\left(-x_1-x_2+1\right)-1\right)\left(H\left(x_1\right)-1\right) \nonumber
\end{align}
The left hand side is the indicator function of the triangle 
in Figure \ref{fig1}(a),
and each term of the right hand side is the indicator function of
the cone in Figure \ref{fig1}(b).
\ifmyfig
\begin{figure}[hbtp]
\center
\begin{minipage}{0.45\hsize}
\center \includegraphics[width=\hsize]{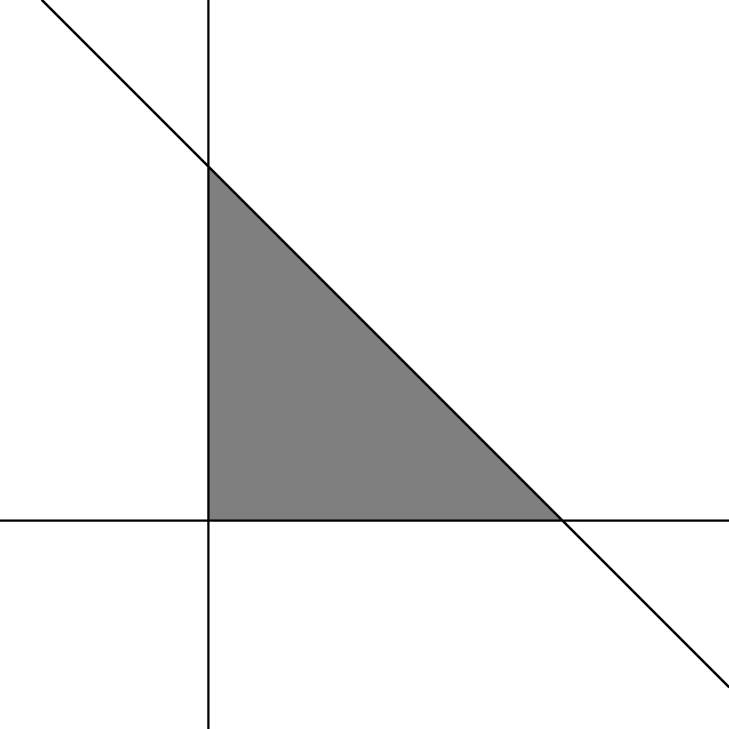}\\(a)
\end{minipage}
\begin{minipage}{0.45\hsize}
\center \includegraphics[width=\hsize]{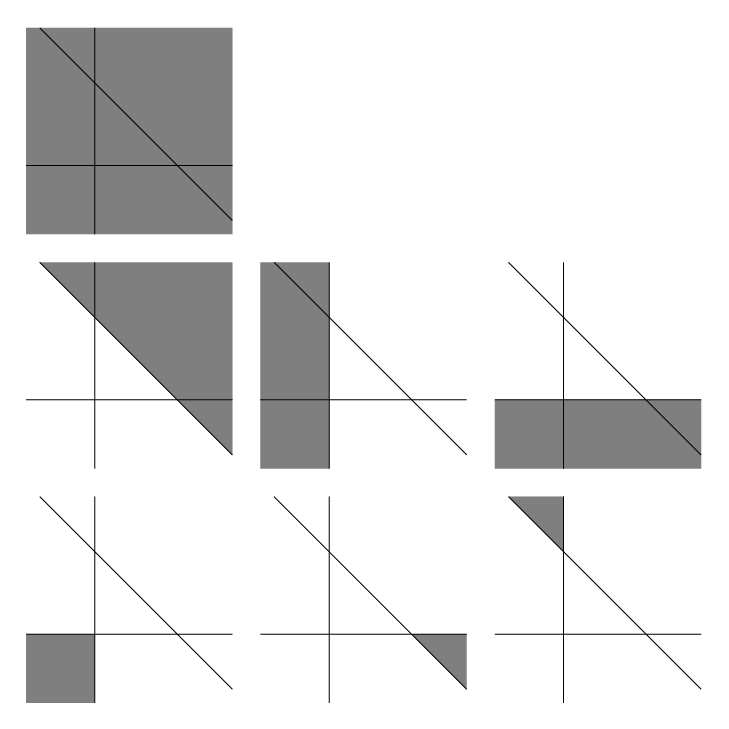}\\(b)
\end{minipage}
\flushleft
    {\small
    Figure \ref{fig1}(a) is the triangle defined by the indicator function on
    the left-hand side of Equation \eqref{2Mar2020a}.
    Figure \ref{fig1}(b) shows the regions represented by the indicator
    functions of the terms on the right-hand side of \eqref{2Mar2020a}.
    }
\center
\caption{An example of the inclusion--exclusion identity}\label{fig1}
\end{figure}
\fi
Let $J=\{1\}$, then the equation \eqref{24} is
\begin{equation}\label{25}
H\left(x_1\right)
H\left(x_2\right)
=1+\left(H\left(x_2\right)-1\right)+\left(H\left(x_1\right)-1\right)
\end{equation}
and this equation holds on $V=\{x\in\mathbf R^2\mid -x_1-x_2+1=0\}$.
On the hyperplane $V$, the left hand side of \eqref{25} expresses 
the indicator function of an edge of the triangle, 
and each term of the right hand side of \eqref{25} is 
the indicator function of the lines in Figure \ref{fig2}.
Note that the all of the vanished terms does not include
$H(-x_1-x_2+1)$.
\ifmyfig
\begin{figure}[hbtp]
\center \includegraphics[width=0.5\hsize]{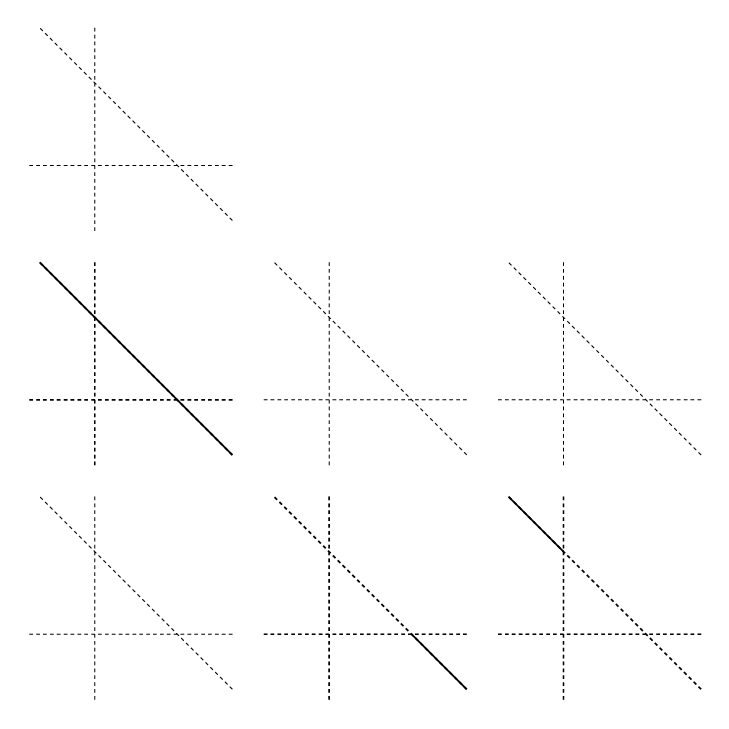}\\
\flushleft
{\small
    The solid line in  Figure \ref{fig2} corresponds to each term on
    the right-hand side of Equation \eqref{25}.
}
%\center
\caption{An example of the inclusion--exclusion identity for a face}\label{fig2}
\end{figure}
\fi
\end{example}

For use in the next section, we extend Proposition \ref{11}.
We introduce the following notation.
For parameters $(a,b)\in\mathbf R^{d\times n}\times\mathbf R^n$, let
\begin{align*}
H_j(a,b)
&=\left\{x\in\mathbf R^d \mid \sum_{i=1}^da_{ij}x_i+b_j\geq 0\right\}
\quad \left(1\leq j\leq n\right),\\
\mathcal H(a,b)
&=\left\{ H_1(a,b),\dots, H_n(a,b)\right\},\\
P(a,b)
&= \bigcap_{j=1}^n H_j(a,b),\\
F_j(a,b)
&= \left\{ x\in\mathbf R^d\mid \sum_{i=1}^da_{ij}x_i+b_j=0\right\}\cap P(a,b),\\
\mathcal F(a,b)
&= \left\{ J\subset [n] \mid \bigcap_{j\in J} F_j(a,b)\neq\emptyset\right\}.
\end{align*}
The following lemma holds.
\begin{lemma}\label{12}
Suppose the polyhedron $P$ is in general position.
Then, there exists a neighborhood $U$ of the parameter 
$(\tilde a,\tilde b)\in\mathbf R^{d\times n}\times\mathbf R^n$
which satisfies the following:
for any parameter $(a,b)$ in $U$, $P(a,b)$ is in general position and 
$\mathcal F(a,b)=\mathcal F$ holds.
\end{lemma}
\begin{proof}
We put $a_{i0}=0\,(1\leq i\leq d)$, $b_0=0$, and 
\begin{align*}
\hat F_j(a,b)
&=\left\{
(x_0,x)\in\mathbf R\times\mathbf R^d 
\mid
\begin{array}{c} 
\sum_{i=1}^d a_{ij}x_i+b_j= 0,\\
\sum_{i=1}^d a_{ik}x_i+b_k\geq 0\\
(0\leq k\leq n)
\end{array}
\right\}
\quad (0\leq j\leq n)\\
\hat{\mathcal F}(a,b)
&=\left\{
J\subset\{0,1,\dots,n\}\mid \bigcap_{j\in J}\hat F_j(a,b)\neq\emptyset
\right\}.
\end{align*}
By Theorem 23 in \cite{Koyama2015}, the set 
$$
U:=\left\{
(a,b)\in\mathbf R^{d\times n}\times R^n \mid
\begin{array}{c}
\text{$P(a,b)$ is in general position},\\
\hat{\mathcal F}(a,b) = \hat{\mathcal F}(\tilde a,\tilde b)
\end{array}
\right\}
$$
is a neighborhood of the point 
$(\tilde a,\tilde b)\in\mathbf R^{d\times n}\times R^n$.
Consider arbitrary $(a,b)\in U$,
from Corollary 19 in \cite{Koyama2015}, we have
$\mathcal F(a,b)=\mathcal F(\tilde a,\tilde b)=\mathcal F$
from $\hat{\mathcal F}(a,b) = \hat{\mathcal F}(\tilde a,\tilde b)$.
By Lemma 22 in \cite{Koyama2015},
all facets of  $P(a,b)$ are given by 
$F_j(a,b)\,\left(\{j\}\in \mathcal F(a,b)\right)$. 
The equation $\mathcal F(a,b)=\mathcal F$ implies 
$\{j\}\in \mathcal F(a,b)$ for all $j\in[n]$.
Consequently, $\mathcal H(a,b)$ is the bounding half-spaces for $P(a,b)$
and $P(a,b)$ is in general position.
\end{proof}
Finally, we have the following.
\begin{theorem}\label{14}
Suppose the polyhedron $P$ is in general position and $J\in\mathcal F$.
There exists a neighborhood $U$ of 
$(\tilde a,\tilde b)\in\mathbf R^{d\times n}\times\mathbf R^n$
such that for any $(a,b)\in U$ and $x\in V(J, a, b)$, we have 
\begin{equation}\label{21}
\prod_{j\in [n]\backslash J} H\left(\sum_{i=1}^d a_{ij}x_i+ b_j\right)
=\sum_{F\in\mathcal F_J}\prod_{j\in F\backslash J}
\left(H\left(\sum_{i=1}^d a_{ij}x_i+ b_j\right)-1\right).
\end{equation}
\end{theorem}
\begin{proof}
Let $U$ be a neighborhood of $(\tilde a, \tilde b)$ in Lemma \ref{12}.
Then the polyhedron $P(a,b)$ is in general position.
By Lemma 22 in \cite{Koyama2015}, the abstract simplicial complex 
associated with $P(a,b)$ is equivalent to $\mathcal F(a,b)$.
The equation $\mathcal F(a,b)=\mathcal F$ implies $J\in\mathcal F(a,b)$.
Hence, we can apply Proposition \ref{11} which gives
$$
\prod_{j\in [n]\backslash J} H\left(\sum_{i=1}^d a_{ij}x_i+ b_j\right)
=\sum_{F\in\mathcal F_J(a,b)}\prod_{j\in F\backslash J}
\left(H\left(\sum_{i=1}^d a_{ij}x_i+ b_j\right)-1\right)
$$
for any $x\in V(J, a, b)$.
Here, we put $\mathcal F_J(a,b)=\{F\subset\mathcal F(a,b)\mid J\subset F\}$.
Since $\mathcal F(a,b)=\mathcal F$ implies 
$\mathcal F_J(a,b)=\mathcal F_J$, we thus have equation \eqref{21}.
\end{proof}

\section{Derivatives of the Probability Content}\label{sec:diff}
In this section we derive an expression for the function $g^J(a,b)\,(J\in\mathcal F)$,
which is a derivative of the function $\varphi(a,b)$ defined by the probability 
content of the polyhedron $P$ for the multivariate normal distribution.
We then show that the function $g^J(a,b)$ can be expressed as 
an integral on the affine subspace \eqref{h-plane}.
For simplicity, we put 
\begin{equation}\label{6}
\varphi_F(a,b)=
\int_{\mathbf R^d}
\exp\left(-\frac{1}{2}\sum_{i=1}^dx_i^2\right)
\prod_{j\in F} H(-f_j(a,b,x))
dx.
\end{equation}
Then, the function $\varphi(a,b)$ in \eqref{eq:varphi} can be written as 
$$
\varphi(a,b)
 = \sum_{F\in\mathcal F}\frac{(-1)^{|F|}}{(2\pi)^{d/2}}\varphi_F(a,b).
$$
In order to obtain expressions for the function $g^J(a,b)$, 
we first consider $\pd{b}^J\bullet\varphi_F(a,b)$ for $F\in\mathcal F$.
Consider the following case:
\begin{lemma}\label{lem:1}
Let $1\leq p\leq q \leq d$, $J=\left\{1,\dots, p\right\}$, and 
$F=\left\{1,\dots, p,\dots,q\right\}$.
Suppose that parameter $a$ satisfies $a_{ij}=0\,(p<i\leq d,\,1\leq j \leq p)$
and $\alpha_F(a)=\left(\sum_{k=1}^da_{ki}a_{kj}\right)_{i,j\in F}$ is a regular
matrix, i.e., 
$$
a=(a_{ij})=
\begin{pmatrix}
a_{11}&\cdots&a_{1p}&a_{1(p+1)}&\cdots&a_{1q}&*&\cdots&*\\
\vdots& & \vdots &\vdots& & \vdots & \vdots& & \vdots\\
a_{p1} & \cdots & a_{pp}&a_{p(p+1)}&\cdots&a_{pq}   &  *   & \cdots & *   \\
0 & \cdots     &  0  &a_{(p+1)(p+1)}&\cdots&a_{(p+1)q} & *   & \cdots & *\\   
\vdots & & \vdots & \vdots   & & \vdots \\   
0 & \cdots     &  0  &a_{d(p+1)}&\cdots&a_{dq} & *   & \cdots & *
\end{pmatrix}
$$
and the vectors $a_1,\dots,a_p,\dots,a_q$ are linearly independent.
Then, the function $\pd{b}^J\bullet\varphi_F(a,b)$ is equal to the integral
\begin{equation}\label{4}
\frac{(-1)^{|J|}}{\sqrt{|\alpha_J(a)|}}
\int_{V(J,a,b)} 
\exp\left(-\frac{1}{2}\sum_{i=1}^dx_i^2\right)
\prod_{j\in F\backslash J} H(-f_j(a,b,x))
\mu(dx).
\end{equation}
Here, $\mu$ is the volume element of the affine subspace $V(J,a,b)$.
\end{lemma}
\begin{proof}
Let a $d\times d$ matrix $U=(u_{ij})$ be 
$$
U=
\begin{pmatrix}
a_{11} & \cdots & a_{1p}   &  0   & \cdots & 0   \\
\vdots&        & \vdots & \vdots& & \vdots\\
a_{p1} & \cdots & a_{pp}   &  0   & \cdots & 0   \\
0 & \cdots     &  0   & 1   & & \\   
\vdots & & \vdots &   &\ddots &  \\   
0 & \cdots     &  0   &    &  & 1
\end{pmatrix}.
$$
The elements of $U$ can be written as 
$$
u_{ij}=
\begin{cases}
a_{ij} & (1\leq j\leq p)\\
\delta_{ij} & (p<j\leq d)
\end{cases}.
$$
Here, $\delta_{ij}$ is Kronecker's delta.
As the vectors $a_1,\dots,a_p$ are linearly independent,
the matrix $U$ is regular, and we have $|U|^2=|\alpha_J(a)|$.
We denote the inverse matrix of $U$ by $U^{-1}=(u^{ij})$.
Consider a transformation of variables
$
y_j = \sum_{i=1}^d u_{ij}x_i\,(1\leq j\leq d).
$
By the relationships
\begin{align*}
x_i&=
\begin{cases}
\sum_{k=1}^p u^{ki}y_k & (1\leq i\leq p)\\
y_i                 & (p+1\leq i\leq d)
\end{cases},\\
y_j&=\sum_{i=1}^da_{ij}x_i+b_j\quad (1\leq j\leq p),
\end{align*}
the integral $\varphi_F(a,b)$ can be written as 
\begin{align*}
&
\frac{1}{\sqrt{|\alpha_J(a)|}}
\int_{\mathbf R^d}
e^{-\frac{1}{2}\sum_{i=1}^p\left(\sum_{k=1}^pu^{ki}y_k\right)^2-\frac{1}{2}\sum_{i=p+1}^dy_i^2}
\prod_{j=1}^p H\left(-y_j-b_j\right)\\
&\quad\times
\prod_{j=p+1}^q H\left(
-\sum_{i=1}^p\sum_{k=1}^p a_{ij}u^{ki}y_k-\sum_{i=p+1}^da_{ij}y_i-b_j
\right)
dy_1\dots dy_d\\
&=
\int_{-\infty}^{-b_1}\cdots\int_{-\infty}^{-b_p}
G(y_1,\dots, y_p; a, b, U)dy_p\dots dy_1.
\end{align*}
Here, we put 
\begin{align*}
&G(y_1,\dots, y_p; a, b, U)\\
&=\frac{e^{-\frac{1}{2}\sum_{i=1}^p\left(\sum_{k=1}^pu^{ki}y_k\right)^2}}{\sqrt{|\alpha_J(a)|}}
\int_{-\infty}^\infty\cdots\int_{-\infty}^\infty
e^{-\frac{1}{2}\sum_{i=p+1}^dy_i^2}
\\&\quad\times
\prod_{j=p+1}^q H\left(
-\sum_{i=1}^p\sum_{k=1}^pa_{ij}u^{ki}y_k-\sum_{i=p+1}^da_{ij}y_i-b_j
\right)
dy_{p+1}\dots dy_d.
\end{align*}
Since $G(y_1,\dots, y_p; a, b, U)$ is a continuous function
with respect to $y_1,\dots, y_p$,
we have 
$\pd{b}^J\bullet\varphi_F(a,b) = (-1)^pG(-b_1,\dots, -b_p; a, b, U)$.

When $p=d$, we have
$$
\pd{b}^J\bullet\varphi_F(a,b)
=
\frac{(-1)^d}{\sqrt{|\alpha_J(a)|}}
e^{-\frac{1}{2}\sum_{i=1}^d\left(\sum_{k=1}^d u^{ki}b_k\right)^2}
$$
Since $-(U^{-1})^\top b$ is the unique point in $V(J,a,b)$, 
$\pd{b}^J\bullet\varphi_F(a,b)$ equals to \eqref{4}.

Suppose $p\neq  d$,
and define a mapping $\psi(x)=(y_{p+1},\dots,y_d)$ for
the affine subspace $V(J,a,b)$ to $\mathbf R^{d-p}$ by $y_j = x_j\,(p+1\leq j\leq d)$,
then this mapping is a local coordinate system on $V(J,a,b)$. 
At this coordinate, the functions $x_1,\dots,x_d$ on $V(J,a,b)$ can be written as 
$$
x_i =
\begin{cases}
-\sum_{k=1}^p u^{ki}b_k & (1\leq i\leq p),\\
y_i & (p+1\leq i\leq d).
\end{cases}
$$
A Riemannian metric can be written as a tensor product of 
1-forms (see, e.g., \cite{helgason1978}), 
and the Riemannian metric induced on $V(J,a,b)$ is 
$\sum_{i=1}^d dx_i\otimes dx_i= \sum_{j=p+1}^d dy_j\otimes dy_j$.
Calculating \eqref{4} with this coordinate, 
we have \\
$ g^J(a,b) = (-1)^p G(-b_1, \dots, -b_p; a, b, U)$.
\end{proof}

We now extend this lemma.
\begin{lemma}\label{lem:2}
Let $1\leq p\leq q\leq d$, $J=\left\{1,\dots, p\right\}$, and 
$F=\left\{1,\dots, p,\dots,q\right\}$.
Suppose $\alpha_F(a)$ is a regular matrix. Then, the function 
$\pd{b}^J\bullet\varphi_F(a,b)$ is equal to \eqref{4}.
\end{lemma}
\begin{proof}
It is sufficient to show that this reduces to Lemma \ref{lem:1}.

For a suitable special orthogonal matrix $R$, the $d\times n$ matrix $a':=Ra$
satisfies the condition $a'_{ij}=0\,(p<i\leq d,\,1\leq j\leq p)$.
Since  $\alpha_F(a)=\alpha_F(a')$,
$\alpha_F(a')$ is also a regular matrix by this assumption.
Hence, the parameter $(a',b)$ satisfies the condition of Lemma \ref{lem:1}.

Since the Lebesgue measure is invariant under the action of the special orthogonal
group, we have $\varphi_F(a,b)=\varphi_F(a',b)$ for any $b\in\mathbf R^n$.
Consequently, we have 
$\pd{b}^J\bullet \varphi_F(a,b)=\pd{b}^J\bullet \varphi_F(a',b)$,
and $1/\sqrt{|\alpha_J(a)|}=1/\sqrt{|\alpha_J(a')|}$.

Considering \eqref{4},  we put 
$$
\tilde\varphi_F(a,b)=
\int_{V(J,a,b)} 
\exp\left(-\frac{1}{2}\sum_{i=1}^dx_i^2\right)
\prod_{j\in F\backslash J} H(-f_j(a,b,x))
\mu(dx).
$$
We need to show $\tilde\varphi_F(a,b)=\tilde\varphi_F(a',b)$.
When $p=d$, this relation is trivial.
Suppose that $p<d$.
Take vectors $v_j=(v_{1j},\dots,v_{dj})^\top\,(q+1\leq j\leq d)$ 
such that $a_1,\dots,a_q,v_{q+1},\dots, v_d$ are linearly independent.
Let $U=(u_{ij})$ be a matrix obtained by arranging these vectors,
$$
u_{ij}=
\begin{cases}
a_{ij} & (1\leq j\leq q),\\
v_{ij} & (q+1\leq j\leq d).
\end{cases}
$$
We denote the inverse matrix of $U$, by $U^{-1}=(u^{ij})$.
We define a matrix $U'=(u'_{ij})$ as $U'=RU$, and denote its inverse by $U'^{-1}=(u'^{ij})$.

First, we calculate $\tilde\varphi_F(a,b)$.
If we define a map $\psi(x)=(y_{p+1},\dots,y_d)$ from the affine subspace $V(J,a,b)$
to $\mathbf R^{d-p}$ by $y_j=\sum_{i=1}^d u_{ij}x_i\,(p+1\leq j\leq d)$, then
it is a local coordinate system on $V(J,a,b)$. 
With this coordinate, the function $x_i$ on $V(J,a,b)$ can be written as 
$$
x_i=-\sum_{k=1}^pu^{ki}b_k+\sum_{k=p+1}^du^{ki}y_k
\quad (1\leq i\leq d).
$$
Hence, the Riemannian metric on the affine subspace $V(J,a,b)$ is 
$$
\sum_{i=1}^ddx_i\otimes dx_i 
=
\sum_{k=1}^p\sum_{\ell=1}^p\left(
 \sum_{i=1}^du^{ki}u^{\ell i}
\right)dy_k\otimes dy_\ell.
$$
Let $D$ be the determinant of the matrix 
$\left( \sum_{i=1}^du^{ki}u^{\ell i}\right)_{1\leq k,\ell\leq p}$.
The integral $\tilde\varphi_F(a,b)$ can be written as
$$
\frac{1}{\sqrt{|D|}}
\int_{\mathbf R^{d-p}}
e^{
  -\frac{1}{2}\sum_{i=1}^d
  \left(
    -\sum_{k=1}^p u^{ki}b_k+\sum_{k=p+1}^d u^{ki}y_k
  \right)^2
}
\prod_{j=p+1}^q H(-y_j-b_j)
\prod_{j=p+1}^ddy_j.
$$

Next, we calculate $\tilde\varphi_F(a',b)$.
If we define a map $\psi'(x)=(y_{p+1},\dots,y_d)$ from the affine subspace $V(J,a',b)$
to $\mathbf R^{d-p}$ by $y_j=\sum_{i=1}^d u'_{ij}x_i\,(p+1\leq j\leq d)$,
then it is a local coordinate system on $V(J,a',b)$. 
With this coordinate, the function $x_i$ on $V(J,a',b)$ can be written as 
$$
x_i=-\sum_{k=1}^pu'^{ki}b_k+\sum_{k=p+1}^du'^{ki}y_k
\quad (1\leq i\leq d).
$$
By $U'^{-1}=U^{-1}R^\top$, the Riemannian metric on $V(J,a',b)$ is  
$$
\sum_{i=1}^ddx_i\otimes dx_i 
=
\sum_{k=1}^p\sum_{\ell=1}^p\left(
 \sum_{i=1}^du'^{ki}u'^{\ell i}
\right)dy_k\otimes dy_\ell
=
\sum_{k=1}^p\sum_{\ell=1}^p\left(
 \sum_{i=1}^du^{ki}u^{\ell i}
\right)dy_k\otimes dy_\ell ,
$$
and we have 
\begin{align*}
\sum_{i=1}^d x_i^2
&=
\sum_{i=1}^d
  \left(
    -\sum_{k=1}^p u'^{ki}b_k+\sum_{k=p+1}^d u'^{ki}y_k
  \right)^2\\
&=
\sum_{i=1}^d
  \left(
    -\sum_{k=1}^p u^{ki}b_k+\sum_{k=p+1}^d u^{ki}y_k
  \right)^2.
\end{align*}
Hence we have $\tilde\varphi_F(a',b)=\tilde\varphi_F(a,b)$.
\end{proof}
From Lemma \ref{lem:2}, we have the following.
\begin{lemma}\label{lem:3}
Let $F\in\mathcal F$ and suppose $\alpha_F(a)$ is a regular matrix.
Then, we have
\begin{align*}
&\pd{b}^J\bullet\varphi_F(a,b)\\
&=
\begin{cases}
\frac{(-1)^{|J|}}{\sqrt{|\alpha_J(a)|}}
\int_{V(J,a,b)} 
e^{-\frac{1}{2}\sum_{i=1}^dx_i^2}
\prod_{j\in F\backslash J} H(-f_j(a,b,x))
\mu(dx)
& (J\subset F),\\
0  & (J\not\subset F).
\end{cases}
\end{align*}
\end{lemma}
\begin{proof}
When $J\subset F$, it reduces to Lemma \ref{lem:2} 
since we can assume $J=\left\{1,\dots, p\right\}$, $
F=\left\{1,\dots, p,\dots,q\right\}$, and $1\leq p\leq q\leq d$
without loss of generality.
When $J\not\subset F$, the integral in \eqref{6} does not depend on 
the variables $b_j\,(j\in J\backslash F)$.
Hence, the derivative with respect to $b_j$ is $0$.
\end{proof}
\begin{theorem}\label{23}
Suppose that the polyhedron $P$ is in general position and $J\in\mathcal F$,
then there exists a neighborhood $U$ of the parameter 
$(\tilde a,\tilde b)\in\mathbf R^{d\times n}\times\mathbf R^n$
such that the equation 
\[
g^J(a,b)
= \frac{1}{(2\pi)^{d/2}\sqrt{|\alpha_J(a)|}}
\int_{V(J,a,b)} 
e^{-\frac{1}{2}\sum_{i=1}^dx_i^2}
\prod_{j\in [n]\backslash J} H(f_j(a,b,x))
\mu(dx)
\]
holds for any $(a,b)\in U$.
\end{theorem}
\begin{proof}
By \eqref{eq:varphi} and \eqref{22}, we have 
$$
(2\pi)^{d/2}g^J(a,b)
= \sum_{F\in \mathcal F} \pd{b}^J\bullet
(-1)^{|F|}
\int_{\mathbf R^d}
e^{-\frac{1}{2}\sum_{i=1}^dx_i^2}
\prod_{j\in F} H(-f_j(a,b,x))
dx.
$$
Applying Lemma \ref{lem:3} to each term on the right hand side of 
the above equation, we can show that $(2\pi)^{d/2}g^J(a,b)$ is equal to 
\begin{align*}
&
\frac{1}{\sqrt{|\alpha_J(a)|}}
\int_{V(J,a,b)} 
e^{-\frac{1}{2}\sum_{i=1}^dx_i^2}
\sum_{F\in\mathcal F_J} (-1)^{|F\backslash J|}
\prod_{j\in F\backslash J} H(-f_j(a,b,x))
\mu(dx)\\
&=
\frac{1}{\sqrt{|\alpha_J(a)|}}
\int_{V(J,a,b)} 
e^{-\frac{1}{2}\sum_{i=1}^dx_i^2}
\sum_{F\in\mathcal F_J} 
\prod_{j\in F\backslash J} H(f_j(a,b,x)-1)
\mu(dx).
\end{align*}
Hence, Theorem \ref{14} implies Theorem \ref{23}.
\end{proof}

\section{Holonomic Gradient Method}\label{sec:HGM}
In this section, we discuss the computation of the probability content
of a polyhedron with a multivariate normal distribution
for the case where the polyhedron is in general position
and bounded, and the case where the polyhedron is a simplicial cone.

\subsection{The Bounded Case}%\label{17}
Let us consider the case where the polyhedron $P$ in general position is bounded.
\begin{lemma}\label{15}
Suppose the polyhedron $P$ is bounded.
Then, the set
\begin{equation}\label{13}
\left\{x\in\mathbf R^d\mid \sum_{i=1}^d\tilde a_{ij}x_i\geq 0\,(1\leq j\leq n)\right\}
\end{equation}
contain only the origin.
\end{lemma}
\begin{proof}
By Proposition 1.12 in \cite{ziegler}, the set \eqref{13} is equal to
$$
\left\{y\in\mathbf R^d\mid x+ty\in P\,\left(x\in P,\, t\geq 0\right)\right\}.
$$
Since $P$ is bounded, this set does not contain any element except the origin.
\end{proof}

\begin{proposition}
Suppose the polyhedron $P$ in general position is bounded.
Then, for $J\in\mathcal F$, we have
\begin{equation}\label{16}
g^J(\tilde a,0)=
\begin{cases}
\frac{1}{\sqrt{|\alpha_J(\tilde a)|}} & (|J|=d)\\
0 & (|J|\neq d)
\end{cases}.
\end{equation}
\end{proposition}
\begin{proof}
Calculating the left hand side, we have
\begin{align*}
&(2\pi)^{d/2}g^J(\tilde a,0)\\
&= \lim_{t\rightarrow +0}(2\pi)^{d/2}g^J(\tilde a,tb)\\
&= \lim_{t\rightarrow +0}
\frac{1}{\sqrt{|\alpha_J(\tilde a)|}}
\int_{V(J,\tilde a,tb)} 
e^{-\frac{1}{2}\sum_{i=1}^dx_i^2}
\prod_{j\in [n]\backslash J} H(f_j(\tilde a,tb,x))
\mu(dx)\\
&= 
\frac{1}{\sqrt{|\alpha_J(\tilde a)|}}
\int_{V(J,\tilde a,0)} 
e^{-\frac{1}{2}\sum_{i=1}^dx_i^2}
\prod_{j\in [n]\backslash J} H(f_j(\tilde a,0,x))
\mu(dx). 
\end{align*}
By Lemma \ref{15}, the integral domain is $\{0\}$. Hence we have \eqref{16}.
\end{proof}
Consequently, in order to compute the probability content of $P$
for a multivariate normal distribution, we can take the path of the HGM as 
$$
a(t) = \tilde a,\, b(t) = \tilde tb
\quad (0\leq t\leq 1).
$$
This path does not pass through the singular locus of the Pfaffian equations
\eqref{pfaff1}, \eqref{pfaff2} and \eqref{pfaff3}.
The initial value $g^J(a(0),0)$ at $t=0$ is given explicitly by \eqref{16}.

\subsection{The Simplicial Cone Case}\label{18}
Consider the case where the polyhedron $P$ is a simplicial cone, i.e.,
$n=d$ and the vectors $\tilde a_1,\dots, \tilde a_d$ are linearly independent.
We can assume without loss of generality that $\tilde a$ is an upper triangular 
matrix.
Then define $\gamma(t)=(a(t),b(t))$ by
$$
a(t) = (1-t)\mathrm{diag}(\tilde a_{11},\dots,\tilde a_{dd})+t\tilde a,\, 
b(t)=t\tilde b
\quad(0\leq t\leq 1).
$$
This does not pass through the singular locus of the Pfaffian equation.
The initial value is 
$$
g^J(a(0),b(0))
=\frac{1}{\left|\prod_{j\in J}\tilde a_{jj}\right|}\sqrt{\frac{\pi}{2}}^{d-|J|}.
$$

\section{Tukey--Kramer studentized range}
We denote by $\mu_d$ a probability measure on 
the Euclidean space $\mathbf R^d$ defined by 
$$
\mu_d(A)
=\int_A \exp\left(-\frac{1}{2}\sum_{i=1}^{d-1}x_i^2\right) dx
\quad (A \subset\mathbf R^d).
$$
Take a parameter $t$, and consider a polyhedron
$$
P(t) := 
\left\{
 x\in\mathbf R^d \mid
 |x_i-x_j|\leq t, \, 1\leq i\leq j\leq d
\right\}.
$$
In this section, we apply the HGM to the function
$$
F(t) := \mu_d(P(t)) \quad (t\geq 0).
$$
This is a special case of Tukey--Kramer studentized range statistics.
See, e.g., \cite{naiman-wynn1997} or \cite{naiman-wynn1992}.

We prepare some lemmas.
\begin{lemma}\label{a1}
Let $U=(u_{ij})$ be a $d\times d$ matrix whose entries are
$$
u_{ij}:=
\begin{cases}
\frac{1}{\sqrt{d}} & ( j=d, 1\leq i\leq d)\\
\frac{1}{\sqrt{j(j+1)}} & (1\leq j\leq d-1, 1\leq i\leq j)\\
-\frac{j}{\sqrt{j(j+1)}} & (1\leq j\leq d-1, i=j+1)\\
0 & (1\leq j\leq d-1, j+2\leq i \leq d)
\end{cases}.
$$
Then, $U$ is a special orthogonal matrix.
\end{lemma}
\begin{proof}
By a straight forward calculation, we can show
$\sum_{k=1}^d u_{ki}u_{kj}=\delta_{ij}\,(1\leq i\leq j\leq d)$.
In fact, we have the following calculation.
For $i=j=d$, we have
\begin{align*}
\sum_{k=1}^d u_{ki}u_{kj}
&= \sum_{k=1}^d u_{kd}^2
 = \sum_{k=1}^d \frac{1}{d} 
 = 1.
\end{align*}
\\
For $i=j<d$, we have
\begin{align*}
\sum_{k=1}^d u_{ki}u_{kj}
&= \sum_{k=1}^d u_{ki}^2 
 = \left(\sum_{k= 1 }^i \frac{1}{i(i+1)}\right)
 + \frac{i^2}{i(i+1)}
 = 1.
\end{align*}
\\
For $i<j=d$, we have
\begin{align*}
\sum_{k=1}^d u_{ki}u_{kj}
&= \sum_{k=1}^d u_{ki}u_{kd}
 = \left(\sum_{k=1}^i \frac{1}{\sqrt{i(i+1)}}\cdot\frac{1}{\sqrt{d}}\right)
 -\frac{i}{\sqrt{i(i+1)}}\cdot\frac{1}{\sqrt{d}}
 = 0.
\end{align*}
\\
For $i<j<d$, we have
\begin{align*}
\sum_{k=1}^d u_{ki}u_{kj}
&=\left(\sum_{k=1}^i \frac{1}{\sqrt{i(i+1)}}\cdot\frac{1}{\sqrt{j(j+1)}}\right)
 -\frac{i}{\sqrt{i(i+1)}}\cdot\frac{1}{\sqrt{j(j+1)}}
 = 0.
\end{align*}

Put a $d\times d$ matrix $V=(v_{ij})$ by
$$
v_{ij}:=
\begin{cases}
1 & (i\geq j)\\
0 & (i < j)
\end{cases}.
$$
Then the matrix $VU$ is an upper triangular matrix,
and the $(i,i)$ components are
$$
\sum_{k=1}^d v_{ik}u_{ki}=
\begin{cases}
\sqrt{\frac{i}{i+1}} & ( 1\leq i \leq d-1)\\
\sqrt{d} & ( i = d)
\end{cases}.
$$
In fact, for $i>j$, the $(i,j)$ component of the matrix $VU$ is 
\begin{align*}
\sum_{k=1}^d v_{ik}u_{kj}
&=\sum_{k=1}^i u_{kj}
 =\left(\sum_{k=1}^j u_{kj} \right) + u_{(j+1)j}\\
&=\left(\sum_{k=1}^j \frac{1}{\sqrt{j(j+1)}} \right) - \frac{j}{\sqrt{j(j+1)}}
 =0.
\end{align*}
For $i<d$, the $(i,i)$ component of $VU$ is
\begin{align*}
\sum_{k=1}^d v_{ik}u_{ki}
&=\sum_{k=1}^i u_{ki}
 =\left(\sum_{k=1}^i u_{ki} \right)\\
&=\left(\sum_{k=1}^i \frac{1}{\sqrt{i(i+1)}} \right)
 =\sqrt{\frac{i}{i+1}}.
\end{align*}
The $(d,d)$ component is 
\begin{align*}
\sum_{k=1}^d v_{dk}u_{kd}
&=\sum_{k=1}^d u_{kd}
 =\left(\sum_{k=1}^d \frac{1}{\sqrt{d}} \right)
 =\sqrt{d}.
\end{align*}

By $\det V=1$ and $\det (VU) = 1$, we have $\det U=1$.
Hence, $U$ is a special orthogonal matrix.
\end{proof}

\begin{lemma}\label{a2}
A polyhedron $Q(t)$ defined by the following system of inequalities 
is a simplex in $\mathbf R^{d-1}$.
\begin{align}
\label{a3}
\sqrt{2}y_1 &\geq  0, \\
\label{a4}
-\sqrt{i-1}y_{i-1}+\sqrt{i+1}y_i &\geq  0 \quad (2\leq i \leq d-1),\\
\label{a5}
-\sum_{j=1}^{d-2}\frac{1}{\sqrt{j(j+1)}}y_j-\frac{d}{\sqrt{(d-1)d}}y_{d-1}+t &\geq  0 
\end{align}
\end{lemma}
\begin{proof}
Put a $(d-1)\times(d-1)$ matrix $a=(a_{ij})$ by 
\[%\begin{equation}\label{a6}
a_{ij}:=
\begin{cases}
 \sqrt{j+1} & (i=j,\, 1\leq j\leq d-1)\\
-\sqrt{j} & (i=j+1,\, 1\leq j\leq d-2)\\
0 & (\text{else}) 
\end{cases}.
\]%\end{equation}
The matrix $a$ is regular, and we denote by $a^{-1}=(a^{ij})$
the inverse matrix of $a$.
By a linear transformation $z=ay$,
the system of the inequalities which defines $Q(t)$ can be written as 
\begin{align*}
z_i&\geq 0 \quad (1\leq i\leq d-1), \\
t-\sum_{k=1}^{d-1}a'_kz_k&\geq 0 .
\end{align*}
Here, we put
$$
a'_k:=\sum_{j=1}^{d-2}\frac{a^{jk}}{\sqrt{j(j+1)}}+\frac{da^{(d-1)k}}{\sqrt{(d-1)d}}.
$$
If we have $a'_k\neq 0$ for all $1\leq k\leq d-1$,
then $Q(t)$ is a simplex.

Since the matrix $a$ is lower triangular,
the inverse matrix $a^{-1}$ is also lower triangular.
Obviously, we have
$$
a^{ii} = \frac{1}{\sqrt{i+1}}>0
\quad(1\leq i\leq d-1).
$$
For $i>j$, 
$
\sum_{k=1}^{d-1} a_{ik}a^{kj}
=-\sqrt{i-1}a^{(i-1)j} +\sqrt{i+1}a^{ij}
=0
$
implies 
$$
a^{ij}=\sqrt{\frac{i-1}{i+1}}a^{(i-1)j}>0.
$$
Hence we have $a'_k>0$.
\end{proof}

By the following proposition, the evaluation of the distribution function
$F(t)$ is the evaluation of the region probability of a simplex.
\begin{proposition}
For $t\geq 0$, the following equation holds:
$$
\mu_d(P(t))=d!\cdot \mu_{d-1}(Q(t)).
$$
\end{proposition}
\begin{proof}
Let $S_d$ be the symmetric group of degree $d$.
The polyhedron $P(t)$ can be written as 
$
%P(t)=
\bigcup_{\sigma\in S_d}
\left(
\left\{x\in\mathbf R^d\mid
x_{\sigma(1)}\geq\cdots\geq x_{\sigma(d)}
\right\}
\cap P(t)
\right).
$
By the symmetry of $\mu_d$, we have
\begin{align*}
&
\mu_d\left(
\left\{x\in\mathbf R^d\mid
x_{\sigma(1)}\geq\cdots\geq x_{\sigma(d)}
\right\}
\cap P(t)
\right)\\
&=\mu_d\left(\left\{x\in\mathbf R^d\mid
x_1\geq \cdots\geq  x_d
\right\}
\cap P(t)
\right)\\
&=\mu_d\left(\left\{x\in\mathbf R^d\mid
x_1\geq \cdots\geq  x_d,
\,x_1-x_d<t, 
\right\}
\right)
\end{align*}
and 
$$
\mu_d(P(t))
= 
d! \cdot
\mu_d\left(\left\{x\in\mathbf R^d\mid
x_1\geq \cdots\geq  x_d,
\,x_1-x_d\leq t, 
\right\}\right).
$$
Let $U$ be the special orthogonal matrix in Lemma \ref{a1}.
By the linear transformation $x=Uy$, we have
\begin{align*}
&
\mu_d\left(\left\{x\in\mathbf R^d\mid
x_1\geq \cdots\geq  x_d,
\,x_1-x_d\leq t, 
\right\}\right)\\
&=
\mu_d\left(\left\{y\in\mathbf R^d\mid
(y_1,\cdots,y_{d-1})^\top\in Q(t),\,
y_d\in\mathbf R
\right\}\right)\\
&=
\mu_{d-1}\left(Q(t)\right).
\end{align*}
In fact, by the substitution $x=Uy$,
the inequalities $x_1\geq x_2$, $x_i\geq x_{i+1}\,(2\leq i\leq d-1)$, and $x_d-x_1+t\geq 0$
imply the inequalities \eqref{a3},\eqref{a4},and \eqref{a5}respectively.
\if0
In fact, we have the following calculations.
Substituting $x=Uy$ to $x_1\geq x_2$, we have
$$
\frac{1}{\sqrt{2}}y_1
+\sum_{j=2}^{d-1}\frac{1}{\sqrt{j(j+1)}}y_j
+\frac{1}{\sqrt{d}}y_d
\geq 
-\frac{1}{\sqrt{2}}y_1
+\sum_{j=2}^{d-1}\frac{1}{\sqrt{j(j+1)}}y_j
+\frac{1}{\sqrt{d}}y_d,
$$
then we have
$$
\sqrt{2}y_1 \geq  0.
$$
Substituting $x=Uy$ to $x_i\geq x_{i+1}\,(2\leq i\leq d-1)$, we have
\begin{align*}
-\frac{i-1}{\sqrt{(i-1)i}}y_{i-1}
+\sum_{j=i}^{d-1}\frac{1}{\sqrt{j(j+1)}}y_j
%+\frac{1}{\sqrt{d}}y_d
&\geq 
-\frac{i}{\sqrt{i(i+1)}}y_i
+\sum_{j=i+1}^{d-1}\frac{1}{\sqrt{j(j+1)}}y_j
%+\frac{1}{\sqrt{d}}y_d
\\
+\frac{1}{\sqrt{d}}y_d
&\quad\quad\quad\quad\quad\quad\quad\quad\quad
+\frac{1}{\sqrt{d}}y_d\\
-\frac{i-1}{\sqrt{(i-1)i}}y_{i-1}
+\frac{i+1}{\sqrt{i(i+1)}}y_i
&\geq 0
\\
-\sqrt{(i-1)}y_{i-1}
+\sqrt{(i+1)}y_i
&\geq 0.
\end{align*}
Substituting $x=Uy$ to $x_d-x_1+t\geq 0$, we have
\begin{align*}
\left(-\frac{d-1}{\sqrt{(d-1)d}}y_{d-1}+\frac{1}{\sqrt{d}}y_d\right)
-\left(
\sum_{j=1}^{d-1}\frac{1}{\sqrt{j(j+1)}}y_j
+\frac{1}{\sqrt{d}}y_d
\right)
+t&\geq 0\\
-\frac{d}{\sqrt{(d-1)d}}y_{d-1}
-\sum_{j=1}^{d-2}\frac{1}{\sqrt{j(j+1)}}y_j
+t&\geq 0.
\end{align*}
\fi
\end{proof}

Applying Theorem 16 in \cite{Koyama2015} and Theorem \ref{23}
to the polyhedron $Q(t)$,
we can obtain a system of differential equations for 
the distribution function $F(t)$.
\begin{proposition}%\label{a7}
The distribution function $F(t)$ satisfies 
the following system of the differential equations:
\begin{align*}
\frac{dF}{dt}
&=F_{11},\\
\frac{dF_{k\ell}}{dt}
&=
-t\frac{k\ell}{k+\ell} F_{k\ell}
+\frac{\ell}{k+\ell} F_{(k+1)\ell}
+\frac{k}{k+\ell} F_{k(\ell+1)}\\
&\quad (k+\ell\leq d,\,k\geq 1,\,\ell\geq 1)\\
F_{k\ell} 
&= 0
\quad (k+\ell>d,\,k\geq 1,\,\ell\geq 1).
\end{align*}
When $t=0$, the initial values of the functions $F(t)$ and $F_{k\ell}(t)$ are
\begin{align*}
F(0)&= 0,\\
F_{k\ell}(0)&=\frac{d!}{(2\pi)^{(d-1)/2}\sqrt{d}} 
&&(k+\ell=d,\,k\geq 1,\,\ell\geq 1),\\
F_{k\ell}(0)&= 0
&&(k+\ell<d,\,k\geq 1,\,\ell\geq 1).
\end{align*}
\end{proposition}
\begin{proof}
Obviously, the simplex $Q(t)$ is defined by the fallowing inequalities:
\begin{align}
\label{a8}
\sqrt{2}y_1 
&\geq  0, \\
\label{a9}
-\sqrt{\frac{i-1}{i}}y_{i-1}+\sqrt{\frac{i+1}{i}}y_i 
&\geq  0 \quad (2\leq i \leq d-1),\\
\label{a10}
-\sum_{j=1}^{d-2}\frac{1}{\sqrt{j(j+1)}}y_j-\frac{d}{\sqrt{(d-1)d}}y_{d-1}+t &\geq  0.
\end{align}
We apply Theorem 16 in \cite{Koyama2015} to the system of inequalities 
\eqref{a8},\eqref{a9}, and \eqref{a10}.
By Lemma \ref{a2},
the abstract simplicial complex associated with $Q(t)$
is $\mathcal F=\{J\subset[d]\mid J\neq [d]\}$.
Put a $(d-1)\times d$ matrix $a=(a_{ij})$ by
$$
a_{ij}=
\begin{cases}
 \sqrt{(i+1)/i}   & (1\leq i\leq d-1, j=i)\\
-\sqrt{i/(i+1)}   & (1\leq i\leq d-2, j=i+1)\\
-\sqrt{1/(i^2+i)} & (1\leq i\leq d-2, j=d)\\
-\sqrt{d/(d-1)}   & (i=d-1,j=d)\\
0 & (\text{else})
\end{cases},
$$
and let $b$  be a vector with length $d$ given by $(0, \dots, 0, 1)^\top$.
Put $\alpha=(\alpha_{ij}):a^\top a$,
i.e., $\alpha_{ij}=\sum_{k=1}^{d-1}a_{ki}a_{kj}$, 
then the components of the matrix $\alpha$ are
$$
\alpha_{ij}=
\begin{cases}
2 & (i=j)\\
-1 & (i-j-1\in d\mathbf Z)\\
-1 & (i-j+1\in d\mathbf Z)\\
0 & (\text{else})
\end{cases}.
$$
Define a map $\gamma:[0,\infty)\rightarrow \mathbf R^{(d-1)\times d}\times R^d$
by $\gamma(t):=(a,tb)$, 
then the derivative of $F_J(t):=g^J(\gamma(t))\,(J\in\mathcal F)$ is 
$$
\frac{dF_J}{dt}=
\begin{cases}
F_{J\cup\{d\}} & (d\notin J)\\
-\sum_{k\in J}\alpha^{dk}_J
\left(\delta_{kd} t F_J+\sum_{\ell\in J^c} \alpha_{k\ell} F_{J\cup \ell}\right)
& (d\in J)
\end{cases}.
$$
For a subset $J\subset [d]$, we put 
$\alpha_J:=(\alpha_{ij})_{i,j\in J}$, and 
$\alpha_J^{-1}:=(\alpha^{ij}_J)_{i,j\in J}$.
%For $1\leq k\leq d-1$, $1\leq \ell\leq d-k$, we put
For $1\leq k,\ell\leq d$, we put
$$
J(k,\ell)
:= \left\{i\in\mathbf N\mid 1\leq i\leq k-1 \right\}
\cup
\left\{i\in\mathbf N\mid d+1-\ell\leq i\leq d \right\}
$$
and $F_{k\ell}(t):= F_{J(k,\ell)}(t)$.
Obviously, we have
$dF/dt=dF_\emptyset/dt=F_{\{d\}}=F_{11}$
and
$F_{k(d-k+1)}(t)=F_{[d]}(t)=0$ for $1\leq k\leq d$.

For $1\leq k\leq d-1$ and $1\leq \ell \leq d-k$,
we have
\begin{align*}
\frac{dF_{k\ell}}{dt}
&=
-\sum_{j\in J(k,\ell)}\alpha^{dj}_{J(k,\ell)}
\left(\delta_{jd} t F_{k\ell}+\sum_{i=k}^{d-\ell}\alpha_{ji} F_{J(k,\ell)\cup \{i\}}\right)
\\&=
-\sum_{j\in J(k,\ell)}\alpha^{dj}_{J(k,\ell)}
\left(
   \delta_{jd} t F_{k\ell}
  +\alpha_{jk} F_{J(k,\ell)\cup \{k\}}
  +\alpha_{j(d-\ell)} F_{J(k,\ell)\cup \{d-\ell\}}
\right)
%\\&=
%-t\alpha^{dd}_{J(k,\ell)}F_{k\ell}
%-\sum_{j\in J(k,\ell)}\alpha^{dj}_{J(k,\ell)}
%\left(
%\right)
\\&=
-t\alpha^{dd}_{J(k,\ell)}F_{k\ell}
+\alpha^{d(k-1)}_{J(k,\ell)}F_{(k+1)\ell}
+\alpha^{d(d-\ell+1)}_{J(k,\ell)}F_{k(\ell+1)}.
\end{align*}
Here, we put $\alpha^{d0}_{J(k,\ell)}:=\alpha^{dd}_{J(k,\ell)}$.
Since we have
\begin{align}
\label{a12}
\alpha^{dd}_{J(k,\ell)}&= k\ell/(k+\ell) , & %\\
%\label{a13}
\alpha^{d(k-1)}_{J(k,\ell)}&=\ell/(k+\ell), & %\\
%\label{a14}
\alpha^{d(d-\ell+1)}_{J(k,\ell)}&=k/(k+\ell),
\end{align}
we have the system of the differential equations in the proposition.
Our proof for the equations \eqref{a12} is tedious, and we leave it in appendix.

By Theorem \ref{23}, $F_{ij}(t)$ can be written as 
$$
\frac{d!}{(2\pi)^{(d-1)/2}\sqrt{|\alpha_{J(k,\ell)}|}}
\int_{V(J(k,\ell),a,tb)} 
e^{-\frac{1}{2}\sum_{i=1}^{d-1} x_i^2}
\prod_{j\in [d]\backslash J} H(f_j(a,tb,x))
\mu(dx).
$$
Since $Q(0)=\{0\}$, 
the domain of the above integral equals to $\{0\}\subset\mathbf R^{d-1}$ when $t=0$.
The equation
\begin{equation} \label{a11}
\det(\alpha_{J(k,\ell)}) = k+\ell
\end{equation}
implies the formulae for the initial values in the proposition.
We also leave our proof for \eqref{a11} to appendix.
\end{proof}

\section{Numerical Experiments}\label{sec:NE}
In this section we  compare the performance of our HGM method with a Monte Carlo simulation method.
In the Monte Carlo simulation method, we used the computer system {\tt R}
\cite{R}.
The programs and the raw data of our numerical experiments are obtained at 
\url{http://github.com/tkoyama-may10/simplex/}.
Our program is compiled by the GNU
C compiler version 6.3.0.
We performed the experiments on an
Intel(R)Core(TM) i5-7200U CPU @ 2.50GHz 2.71Gz
with 4.00GB RAM,
running Linux on virtual machine.

First, we evaluate the probability contents of simplices.
For an integer $d\geq 2$, we define polyhedra $P_d$ and $Q_d$ as 
\begin{align*}
P_d&=\left\{
x\in\mathbf R^d\mid
\begin{array}{c}
x_i+\frac{\sqrt{d}}{2}\geq 0\,(1\leq i\leq d),\\
-x_1-\dots -x_d+\frac{\sqrt{d}}{2}\geq 0 
\end{array}
\right\}, \\
Q_d&=\left\{
x\in\mathbf R^d\mid
\begin{array}{c}
x_i-\frac{\sqrt{d}}{2}\geq 0\,(1\leq i\leq d),\\
-x_1-\dots -x_d+\frac{(2d+1)\sqrt{d}}{2}\geq 0 
\end{array}
\right\}
\end{align*}
Both $P_d$ and $Q_d$ are simplices, and they are in general position and bounded. 
In the Monte Carlo method, we generated $1,000,000$ points from a normal 
distribution and computed the fraction of the sample that fell into simplices.

The probability contents obtained by the HGM and Monte Carlo methods are given in Tables \ref{tab1} and \ref{tab3}.
We also show the computational times for the HGM in the tables.

\ifmytab
%% tab1 was generated by monte_carlo.R
\begin{table}[htbp]
\caption{The probability content of $P_d$ as obtained
 by the HGM and Monte Carlo methods.}\label{tab1}
\begin{center}
\begin{tabular}{cccc}
$d$ & HGM & time of HGM(s) & MC\\
\hline
2 & 0.285205 & 0.00 & 0.285483\\
3 & 0.251995 & 0.00 & 0.250978\\
4 & 0.241744 & 0.00 & 0.241334\\
5 & 0.242724 & 0.01 & 0.242846\\
6 & 0.250219 & 0.05 & 0.250527\\
7 & 0.261920 & 0.20 & 0.261850\\
8 & 0.276510 & 0.65 & 0.277142\\
9 & 0.293138 & 1.92 & 0.293659\\
10 & 0.311198 & 5.73 & 0.310761\\
\if0 % old result
2 & 0.285205 & 0.00 & 0.285483\\
3 & 0.251995 & 0.00 & 0.250978\\
4 & 0.241744 & 0.01 & 0.241334\\
5 & 0.242724 & 0.03 & 0.242846\\
6 & 0.250219 & 0.10 & 0.250527\\
7 & 0.261920 & 0.32 & 0.261850\\
8 & 0.276510 & 1.04 & 0.277142\\
9 & 0.293138 & 3.16 & 0.293659\\
10 & 0.311198 & 9.53 & 0.310761\\
\fi
\hline
\end{tabular}
\end{center}
\end{table}
\fi

\ifmytab
%% tab3 was generated by monte_carlo.R
\begin{table}[htbp]
\caption{The probability content of $Q_d$ as obtained by the HGM and Monte Carlo methods.}
\label{tab3}
\begin{center}
\begin{tabular}{cccc}
$d$& HGM & time of HGM(s) & MC\\
\hline
2 & 5.1758e-02 & 0.00 & 5.1856e-02\\
3 & 7.0235e-03 & 0.00 & 6.9250e-03\\
4 & 6.3101e-04 & 0.00 & 6.4200e-04\\
5 & 3.9722e-05 & 0.01 & 3.8000e-05\\
6 & 1.8042e-06 & 0.05 & 1.0000e-06\\
7 & 5.9878e-08 & 0.20 & 0.0000e+00\\
8 & 1.4799e-09 & 0.54 & 0.0000e+00\\
9 & 1.1393e-11 & 1.38 & 0.0000e+00\\
10 & 1.2861e-11 & 3.48 & 0.0000e+00\\
\if0 % old result
2 & 5.1758e-02 & 0.00 & 5.1856e-02\\
3 & 7.0235e-03 & 0.00 & 6.9250e-03\\
4 & 6.3101e-04 & 0.01 & 6.4200e-04\\
5 & 3.9722e-05 & 0.02 & 3.8000e-05\\
6 & 1.8042e-06 & 0.10 & 1.0000e-06\\
7 & 5.9878e-08 & 0.32 & 0.0000e+00\\
8 & 1.4799e-09 & 0.85 & 0.0000e+00\\
9 & 1.1393e-11 & 2.24 & 0.0000e+00\\
10 & 1.2861e-11 & 5.74 & 0.0000e+00\\
\fi
\hline
\end{tabular}
\end{center}
\end{table}
\fi

Note that the accuracy of the Monte Carlo method is low when the probability content of $Q_d$ is very small, and for dimensions greater than $6$, the Monte Carlo method could not
evaluate the probability content of $Q_d$. 
The number of points was not enough to evaluate the probability.

Next, we estimate the probability content of a simplicial cone.
For an integer $d\geq 2$, we define a polyhedron $C_d$ as 
\begin{align*}
C_d &= \left\{
x\in\mathbf R^d\mid
\sum_{i=1}^d a_{ij}x_i + \frac{\sqrt{d}}{2}\geq 0\,(1\leq j\leq d)
\right\}, \\ 
a_{ij}&=
\begin{cases}
(i+j)/100 & (i<j)\\
1         & (i=j)\\
0         & (i>j)
\end{cases}.
\end{align*}
We can evaluate the multivariate normal probability of a simplicial cone
using the method presented in Subsection \ref{18}.
Table \ref{tab2} shows the probability content of $C_d$ evaluated
by the HGM and Monte Carlo methods, 
and the computational times for the HGM.
In the Monte Carlo method, we generated $1,000,000$ samples.
\ifmytab
%% tab2 was generated by monte_carlo.R
\begin{table}[htbp]
\caption{Results of HGM for multivariate normal probabilities of simplicial cones}
\label{tab2}
\begin{center}
\begin{tabular}{cccc}
dim& HGM & time of HGM(s) & MC\\
\hline
2 & 0.580822 & 0.00 & 0.581123\\
3 & 0.532131 & 0.00 & 0.531827\\
4 & 0.512854 & 0.03 & 0.511994\\
5 & 0.509868 & 0.29 & 0.509142\\
6 & 0.516602 & 2.21 & 0.516291\\
7 & 0.529243 & 14.80 & 0.527087\\
8 & 0.545340 & 82.20 & 0.543322\\
9 & 0.563203 & 431.00 & 0.560342\\
10 & 0.581630 & 2146.00 & 0.578026\\
\if %old result
2 & 0.580822 & 0.00 & 0.581123\\
3 & 0.532131 & 0.01 & 0.531827\\
4 & 0.512854 & 0.06 & 0.511994\\
5 & 0.509868 & 0.52 & 0.509142\\
6 & 0.516602 & 4.00 & 0.516291\\
7 & 0.529243 & 25.80 & 0.527087\\
8 & 0.545340 & 147.00 & 0.543322\\
9 & 0.563203 & 771.00 & 0.560342\\
10 & 0.581630 & 3804.00 & 0.578026\\
\fi
\hline
\end{tabular}
\end{center}
\end{table}
\fi
%!!! END !!! DO NOT FORGET ADDING APPENDIX!

\iftrue%\iffalse%

\else

\bibliographystyle{plain}
\bibliography{simplex}

\fi

\appendix
\section{Proof of \eqref{a12} and \eqref{a11}}
\noindent
In the case where $k=1$:\/
The choresky decomposition $L=(L_{ij})_{i,j\in J(1,\ell)}$ of 
$\alpha_{J(1,\ell)}=LL^\top$ is 
$$
L_{ij} =
\begin{cases}
\sqrt{(i+\ell-d+1)/(i+\ell-d)} & (i=j)\\
-\sqrt{(i+\ell-d-1)/(i+\ell-d)} & (i=j+1)\\
0 &(\text{else})
\end{cases}.
$$
Hence, we have
$$
\det\alpha_{J(1,\ell)}
=(\det L)^2
=\prod_{i=d-\ell+1}^d\frac{i+\ell-d+1}{i+\ell-d}=\ell+1=k+\ell.
$$
We denote by $L^{-1}=(L^{ij})_{i,j\in J(1,\ell)}$ the inverse of 
the lower triangular matrix $L$.
The relations
\begin{align*}
L^{dd}L_{dd} &=1,\\
L^{dj}L_{jj}+L^{d(j+1)}L_{(j+1)j} &=0 \quad (d-\ell+1\leq j\leq d-1)
\end{align*}
from $L^{-1}L=I$ imply
$
L^{dj}=(j-d+\ell)/\sqrt{\ell(\ell+1)}
$.
Here, we denote by $I$ the identity matrix.

\if0
\color{blue}
\noindent NOTE:
$$
L^{dd}\sqrt{\frac{\ell+1}{\ell}}=1
$$
implies 
$$
L^{dd} = \sqrt{\frac{\ell}{\ell+1}}=\frac{d-d+\ell}{\sqrt{\ell(\ell+1)}}.
$$
$$
L^{dj}\sqrt{\frac{j+\ell-d+1}{j+\ell-d}}
-\frac{j+1-d+\ell}{\sqrt{\ell(\ell+1)}}
\sqrt{\frac{j+1+\ell-d-1}{j+1+\ell-d}}
=0
$$
implies 
\begin{align*}
L^{dj} 
&= 
\frac{j+1-d+\ell}{\sqrt{\ell(\ell+1)}}
\sqrt{\frac{j+\ell-d}{j+1+\ell-d}}
\sqrt{\frac{j+\ell-d}{j+\ell-d+1}}
=
\frac{j-d+\ell}{\sqrt{\ell(\ell+1)}}
\end{align*}
\color{black}
\fi

The equation $\alpha^{-1}_{J(1,\ell)}=(L^{-1})^\top L^{-1}$ implies
\begin{align*}
\alpha^{dd}_{J(1,\ell)}
&
=L^{dd}L^{dd}
=\ell/(\ell+1)
=\ell/(k+\ell)
,\\
\alpha^{d(d-\ell+1)}_{J(1,\ell)}
&
=L^{dd}L^{d(d-\ell+1)}
=1/(\ell+1)
=k/(k+\ell).
\end{align*}

\medskip\noindent
In the case where $k>1$:\/
The choresky decomposition $L=(L_{ij})_{i,j\in J(1,\ell)}$ of 
$\alpha_{J(1,\ell)}=LL^\top$ is 
$$
L_{ij} =
\begin{cases}
\sqrt{(i+1)/i} & (1\leq i\leq k-1, j=i)\\
-\sqrt{(i-1)/i} & (2\leq i\leq k-1, j=i-1)\\
\sqrt{(i+\ell-d+1)/(i+\ell-d)} & (d-\ell+1 \leq i\leq d-1, j=i)\\
-\sqrt{(i+\ell-d-1)/(i+\ell-d)} & (d-\ell+2\leq i\leq d-1, j=i-1)\\
-\sqrt{1/(j(j+1))} & (i=d, 1\leq j\leq k-1)\\
-\sqrt{(\ell-1)/\ell}& (i=d, j=d-1)\\
\sqrt{(k+\ell)/(k\ell)} & (i=j=d) \\
0 &(\text{else})
\end{cases}.
$$
Hence, we have
$$
\det\alpha_{J(k,\ell)}
=(\det L)^2
=\left(\prod_{i=1}^{k-1}\frac{i+1}{i}\right)
\left(\prod_{i=d-\ell+1}^{d-1}\frac{i+\ell-d+1}{i+\ell-d}\right)
\frac{k+\ell}{k\ell}
=k+\ell.
$$
We denote by $L^{-1}=(L^{ij})_{i,j\in J(1,\ell)}$ the inverse of 
the lower triangular matrix $L$.
The relations
\begin{align*}
L^{dd}L_{dd} &=1,\\
L^{dj}L_{jj}+L^{d(j+1)}L_{(j+1)j} &=0 \quad (d-\ell+1\leq j\leq d-1),\\
L^{d(k-1)}L_{(k-1)(k-1)}+L^{dd}L_{d(k-1)} &=0,\\
L^{dj}L_{jj}+L^{d(j+1)}L_{(j+1)j}+L^{dd}L_{dj} &=0 \quad (1\leq j\leq k-1)
\end{align*}
from $L^{-1}L=I$ imply
$$
L^{dj}=
\begin{cases}
(j+\ell-d)\sqrt{\frac{k}{\ell(k+\ell)}} & (d-\ell+1\leq j\leq d)\\
(-j+k)\sqrt{\frac{\ell}{k(k+\ell)}} & (1\leq j\leq k-1)
\end{cases}.
$$

\if0
\color{blue}
\noindent NOTE:
$$
L^{dd}\sqrt{\frac{k+\ell}{k\ell}}=1
$$
implies 
$$
L^{dd} = \sqrt{\frac{k\ell}{k+\ell}}.
$$
$$
L^{d(d-1)}\sqrt{\frac{\ell}{\ell-1}}
-\sqrt{\frac{k\ell}{k+\ell}}
\sqrt{\frac{\ell-1}{\ell}}
=0
$$
implies 
\begin{align*}
L^{d(d-1)} 
&= 
\sqrt{\frac{k\ell}{k+\ell}}
\frac{\ell-1}{\ell}
=
(\ell-1)\sqrt{\frac{k}{\ell(k+\ell)}}.
\end{align*}
$$
L^{dj}\sqrt{\frac{j+\ell-d+1}{j+\ell-d}}
-(j+1+\ell-d)\sqrt{\frac{k}{\ell(k+\ell)}}
\sqrt{\frac{j+1+\ell-d-1}{j+1+\ell-d}}
=0
$$
implies 
\begin{align*}
L^{dj} 
&=
(j+1+\ell-d)\sqrt{\frac{k}{\ell(k+\ell)}}
\frac{j+\ell-d}{j+1+\ell-d}
=
(j-d+\ell)\sqrt{\frac{k}{\ell(k+\ell)}}
\end{align*}
$$
L^{d(k-1)}\sqrt{\frac{k}{k-1}}
-\sqrt{\frac{k\ell}{k+\ell}}\sqrt{\frac{1}{(k-1)k}}
=0
$$
implies 
$$
L^{d(k-1)} = 
\sqrt{\frac{\ell}{k(k+\ell)}}
$$
$$
L^{dj}\sqrt{\frac{j+1}{j}}
-(-j-1+k)\sqrt{\frac{\ell}{k(k+\ell)}}\sqrt{\frac{j}{j+1}}
-\sqrt{\frac{k\ell}{k+\ell}}\sqrt{\frac{1}{j(j+1)}}
=0
$$
implies 
\begin{align*}
L^{dj} 
&= \frac{(-j-1+k)j+k}{j+1}\sqrt{\frac{\ell}{k(k+\ell)}}
=(-j+k)\sqrt{\frac{\ell}{k(k+\ell)}}
\end{align*}
\color{black}
\fi

The equation $\alpha^{-1}_{J(k,\ell)}=(L^{-1})^\top L^{-1}$ implies
\begin{align*}
\alpha^{d(k-1)}_{J(k,\ell)}
&
=L^{dd}L^{d(k-1)}
=\ell/(k+\ell)
,\\
\alpha^{d(d-\ell+1)}_{J(k,\ell)}
&
=L^{dd}L^{d(d-\ell+1)}
=k/(k+\ell).
\end{align*}

\end{document}